\documentclass[12pt]{amsart}
\usepackage{geometry}                
\usepackage{amsmath}
\usepackage{amscd}
\usepackage{amsthm}
\usepackage{graphicx}
\usepackage{amssymb}
\usepackage{epstopdf}
\usepackage{amssymb,amsbsy,amsmath,amsfonts,amssymb,amscd}
\usepackage{latexsym}
\usepackage{amsxtra}
\usepackage{amscd}
\usepackage{graphics}

\usepackage{epic}

\theoremstyle{plain}
\newtheorem{Thm}{Theorem}

\newtheorem{Prop}[Thm]{Proposition}

\theoremstyle{definition}

\newtheorem{Expl}[Thm]{Example}

\theoremstyle{Remark}
\newtheorem{Rem}[Thm]{Remark}

\usepackage{color}
\input xy
\xyoption{all}

\newcommand\sU{{\mathcal U}}

\newcommand\al{\alpha}

\newcommand\s{\sigma}

\newcommand\ga{\gamma}

\newcommand{\CC}{\ensuremath{\mathbb{C}}}

\newcommand{\ZZ}{\ensuremath{\mathbb{Z}}}

\newcommand{\NN}{\ensuremath{\mathbb{N}}}
\newcommand{\hol}{\ensuremath{\mathcal{O}}}

\newcommand{\ra}{\ensuremath{\rightarrow}}

\numberwithin{equation}{section}

\title[Fujita decomposition ]{ Fujita decomposition over  higher dimensional base}
\author{Fabrizio Catanese and Yujiro Kawamata}
\date{}                                           
\address {Lehrstuhl Mathematik VIII\\
Mathematisches Institut der Universit\"at Bayreuth\\
NW II,  Universit\"atsstr. 30\\
95447 Bayreuth}
\email{fabrizio.catanese@uni-bayreuth.de}
\address {Graduate School of Mathematical Sciences, University of Tokyo,\\
Komaba, Meguro, Tokyo, 153-8914, Japan}
\email{kawamata@ms.u-tokyo.ac.jp}
\begin{document}
\maketitle

\begin{abstract}
We generalize a result of Fujita,  on the decomposition of Hodge bundles 
over curves,  to the case of a higher dimensional base.
\end{abstract}


\section{introduction}

The purpose of this paper is to generalize a result of Fujita, on the decomposition of Hodge bundles 
over curves,  to the case of a  higher dimensional base.
We first consider the direct image sheaf of the relative canonical sheaf for a fibration
whose degeneracy locus (set of critical values) is  contained in a normal crossing divisor and with unipotent 
local monodromies.
It is known that such a sheaf is locally free and numerically semipositive (or {\em nef}) 
by Fujita and Kawamata (\cite{F1}, \cite{K1}).
Moreover if the base space of the fibration is a curve, then Fujita proved that the direct image sheaf 
is a direct sum of an ample vector bundle and a unitary flat  bundle  with respect to 
the natural hermitian metric (\cite{F2}). 
We shall generalize the latter result to the case of a higher dimensional base.

The algebraic statement on the nefness of the Hodge bundle was the starting point of many positivity 
results concerning canonical and pluri-canonical sheaves.
These results were used in the minimal model program.
But the original theorem's statement is more analytic saying that the direct image sheaf carries a 
semi-positive hermitian metric with mild but not algebraic singularities.
We believe that this analytic statement should have more consequences 
(see e.g. Theorem~\ref{sum}).

The following is our first main result:

\begin{Thm}\label{main}  [unipotent monodromies case] 

Let $f: X \to Y$ be a proper surjective morphism from a compact connected 
K\"ahler manifold to a smooth projective variety.
Assume that there are respective simple normal crossing divisors $B$  on $X$ and  $D$ on $Y$ such that 
$f$ is smooth over $Y \setminus D$ and that $f^{-1}(D) = B$ set-theoretically.
Set $X^o : = X \setminus B$, $Y^o : = Y \setminus D$, $f^o : = f \vert_{X^o}$ and $n : = \dim X - \dim Y$. 

Assume that the local monodromies of the local system $R^nf^o_*\mathbf{C}_{X^o}$ around the 
branches of $D$ are unipotent.
Let $V : = f_*\mathcal{O}_X(K_X - f^*K_Y)$ be the direct image sheaf of the relative canonical sheaf.
Then 

(1) $V$ is a locally free sheaf which is numerically semi-positive.

Moreover 

(2) there is a decomposition $V = U \oplus W$, 
an orthogonal direct sum with respect to the natural hermitian metric,
where $U$ is a unitary flat bundle and $W$ is generically ample, i.e., 
$W \vert_C$ is an ample vector bundle for any general curve section $C$ of $Y$.   
\end{Thm}

The first assertion is proved in \cite{F1}  for the case $\dim Y = 1$ and in \cite{K1} generally.
See also \cite{K2}.
The second assertion when $\dim Y = 1$ is the essential part of Fujita's theorem \cite{F2}.
See \cite{cat-dett1}, \cite{cat-dett} for details of proof of Fujita's more general statement.
 We recall that one or both of the two summands $U,W$ can be $=0$; 
in \cite{cat-dett} and \cite{cat-dett2} 
examples are given where $U$ corresponds to a representation 
of $\pi_1(Y)$ of infinite order, and $W \neq 0$, in particular  $V$  is not semi-ample.  

Next we prove the following full fledged analogue of Fujita's  decomposition theorem over a higher dimensional base:

\begin{Thm}\label{full}  [general case]

Let $f: X \to Y$ be a proper surjective morphism from a compact connected 
K\"ahler manifold to a smooth projective variety.
Let $V : = f_*\mathcal{O}_X(K_X - f^*K_Y)$ be the direct image sheaf of the relative canonical sheaf.
Then 

(1) $V$ is a torsion free sheaf which is weakly-positive.

(1') $V$ is a locally free sheaf if there is a normal crossing divisor $D$ such that 
$f$ is smooth over $Y \setminus D$.

Moreover 

(2) there is a decomposition $V = U \oplus W$, an orthogonal direct sum with respect to the natural hermitian metric,
where $U$ is a locally free sheaf 
which is unitary flat with respect to the natural hermitian connection 
and $W$ is generically ample, i.e., 
$W \vert_C$ is an ample vector bundle for any general curve section $C$ of $Y$.   
\end{Thm}

(1) and (1') are proven in \cite{viehweg}  using \cite{K1}. 
There are two ways to prove the local freeness of $V$:\\
(i) reducing to the unipotent monodromy case by using the flatness of the covering of the base as in \cite{viehweg};\\
(ii) as a consequence of the extension of the Hodge filtration to the canonical extension, in the same 
way as in the unipotent monodromy case (cf. \cite{K2}).

\vskip 1pc

Recall that (\cite{F1}, \cite{K1}, resp. \cite{viehweg})
a locally free sheaf $V$  on a projective variety $Y$ is said to be {\em numerically semi-positive}, or {\em nef}, 
if given any curve $C$ and any quotient bundle $Q$ of $V | C$, 
$\text{deg}(Q) \geq 0$; whereas a torsion free sheaf  $V$
on a projective variety $Y$  is  said to be {\em weakly positive} if there  are 
a Zariski open  subset $U$  and an ample divisor $H$ on $Y$ 
such that, for each positive integer $n$, the natural homomorphism
\[
H^0(Y, S^{mn}(V)^{\vee \vee} \otimes \mathcal{O}_Y(mH)) \otimes \mathcal{O}_Y \to 
S^{mn}(V)^{\vee \vee}  \otimes \mathcal{O}_Y(mH)
\]
is surjective on $U$ for some positive integer $m$ (here $V^{\vee \vee} $ is the double dual of $V$).

\vskip 1pc

It makes sense to separate  Theorems \ref{main} and  \ref{full}, 
since from the former follows an orbifold  version
of Fujita's decomposition  theorem in the general case.

This orbifold version  is proved in the final section, 
where we shall compare, through an appropriate diagram, the direct image
sheaves  in the different situations which occur in practice, when one takes first a normal crossing resolution of
the degeneracy locus, and then, as in  \cite{K1}, a finite  cover of the base in order to reduce to the case
of unipotent local monodromies.

We define a flat unitary orbifold sheaf on a smooth variety $Y$ in the 
following way.
Let $D$ be a divisor on $Y$, with irreducible components $D_1, \dots D_k$.
An orbifold unitary representation  $\rho : \pi_1 (Y \setminus D) \ra U(r)$, 
with  orders $m_1, \dots, m_k \in \NN_{\geq 1}$ is defined to be  one such  
that, for a small loop $\ga_j $ around $D_j$, 
the image $\rho (\ga_j) $ has order exactly $m_j$.  
Let $Y^{**}$ be the complement of the locus where $D$ is not normal crossing. 
We define a locally free sheaf $U^{**}$ on 
$Y^{**}$ by taking the upper canonical extension of the flat unitary bundle $\sU$ 
on $Y^0 = Y \setminus D$ associated to the representation $\rho$.
 Then we define a {\em flat unitary orbifold sheaf} $U^{\text{orb}}(\rho)$ 
associated to the orbifold unitary representation  $\rho$ 
as the direct image sheaf $i_* (U^{**})$, where $ i : Y^{**} \ra Y$ is the inclusion. 

\begin{Prop}\label{orb}
In the general case we have an orthogonal splitting  $V = U^{\text{orb}} \oplus W$,
where $W$ is generically ample but where this time  $U^{\text{orb}}$ is an orbifold flat unitary sheaf, 
having vanishing curvature  outside the locus $D$ over which $f$ is not smooth.  

$U^{\text{orb}}$ is a  locally free sheaf  on a big open subset $Y^{oo}$ of $Y$ on which 
$D$ is a normal crossing divisor. 
 On $Y^{oo}$,  $U^{\text{orb}}$  coincides with the upper canonical extension 
(see the next section) corresponding 
 to a unitary representation  $\rho : \pi_1 (Y \setminus D) \ra U(r)$, 
 with $r = \text{rank} \ (U^{\text{orb}})$,
such that, for small loops $\ga_j $ around the divisorial components $D_j$ of $D$, 
the $\rho (\ga_j) $'s have  finite order.  

$U$ is  a locally free summand  of $U^{\text{orb}}$,
corresponding to the irreducible summands $\rho_h$ of $\rho $ such that $\rho_h (\ga_j) =1 $ 
for all  $\ga_j $.
\end{Prop}

Acknowledgement:  This paper grew out of a discussion when the second author visited the first author
at the University of Bayreuth.
The present work took place in the framework  of the 
ERC Advanced grant n. 340258, `TADMICAMT',  and of the JSPS Grant-in-Aid (A) 16H02141.
The second author would also like to thank the National Center for Theoretical Sciences at Taiwan University and
Professor Jungkai Chen for the excellent working conditions. 
   
\section{Proof of Theorem~\ref{main}}

We use the notation of Theorem~\ref{main}.
The first assertion (1), the nefness of $V$, is proved in the following way.
This is not new 
(see the  papers cited in the bibliography, also for more references regarding several assertions made here), 
but we briefly recall the proof because we need the notation used in the argument.

We assume that the degeneracy locus $D$ of $f$ (the set of critical values of $f$) is  contained in 
a normal crossing divisor, but do not assume that 
the local monodromies around $D$ are unipotent, so that the following construction can be used 
also for the proof of Theorem~\ref{full}.

The primitive part $H^o_{\ZZ}$ of the local system $R^nf^o_*\CC_{X^o}$ 
underlies a polarized variation of Hodge structures.
Let $H^o : = H^o_{\ZZ} \otimes_{\CC} \hol_{Y^o}$.
It is a locally free sheaf on $Y^o$ with a flat connection, called the {\em Gauss-Manin connection}, 
whose sheaf of flat sections is $H^o_{\CC}$.
The direct image sheaf $V^o = f^o_*\mathcal{O}_X(K_X - f^*K_Y)$ is identified as the subsheaf $F^n(H^o)$, 
where $F^{\cdot}$ denotes the Hodge filtration.
The cup product along the fibres of $f^o$ determines a polarization of the Hodge structures, 
a non-degenerate but not definite hermitian metric on $H^o$ which is flat with respect 
to the Gauss-Manin connection, but with the property that  the restriction of this metric to $V^o$ is positive definite.

$H^o$ has an extension $H$ as a locally free sheaf on $Y$ called the {\em canonical extension} 
constructed in the following way (\cite{deligne}).
Let $x_j$ be local coordinates on $Y$ such that $D$ is expressed as $\prod_{j=1}^r x_j = 0$, and 
let $T_j$ be the corresponding local monodromy transformations for $1 \le j \le r$.
Then $H$ is locally generated by the following holomorphic sections
\[
\exp (- \frac 1{2\pi i} \sum_{j=1}^r \log T_j \log x_j)v
\]
where the $v$ are multi-valued flat sections of $H^o$, and the matrices $\log T_j$ are defined to be
$\log U_j + \log S_j$ for the decomposition $T_j = U_jS_j$ into unipotent and semisimple part.
We take the branches of the logarithm function such that the eigenvalues of the matrices $\log S_j$ 
belong to the interval $[0,2\pi i)$.
We note that $\log U_j$ is well defined because $U_j - I$ is nilpotent, 
and that the expression is single vlaued even if $v$ is multi-valued.

If the local monodromies are unipotent, 
then the local holomorphic sections of $H$ are characterized, among local holomorphic sections of $H^o$, 
by the condition that the norm of their restrictions to $H^o$ 
grow at most logarithmically along the boundary $D$.
In general, they grow asymptotically as  $\prod_{j=1}^r \vert x_j \vert^{-p_j} \vert \log x_j \vert^{q_j}$ 
for some $p_j,q_j$ such that $0 \le p_j < 1$.
We note that the failure of the unipotency of the $T_j$ 
is reflected in the eigenvalues of the $S_j$, whence the exponents $p_j$ above.
We call these exponents  $p_j$ {\em boundary contributions}.

The Hodge filtration on $H^o$ extends to $H$ in such a way that the $F^p(H)$'s are locally free subbundles.
Moreover we have $V = f_*\mathcal{O}_X(K_X - f^*K_Y) = F^n(H)$. 
This fact holds even if the local monodromies are not unipotent as long as 
$f$ is smooth outside a normal crossing divisor.
In particular $V$ is locally free (see also \cite{K2}).

The Hodge-Riemann bilinear relations and  Griffiths'  transversality (\cite{Griffiths}) 
imply that the curvature of the connection of $V^o$ corresponding to the restricted metric on $V^o$ is semi-positive.
The nefness of $V$ follows from this semi-positivity together with the observation 
that the boundary contributions along $D$ vanish because the metric grows at most logarithmically.

\vskip 1pc

Now we prove the latter assertion (2).
We fix an ample line bundle $L$ on $Y$.
Let $0 \subset N_1 \subset \dots \subset N_m = V$
be a Harder-Narasimhan filtration with respect to $L$;
the $N_i$ are reflexive subsheaves of $V$ such that the subquotients $N_i/N_{i-1}$ are semistable
torsion free sheaves with respect to $L$, 
and the inequalities $\mu(N_i/N_{i-1})> \mu(N_{i+1}/N_i)$ hold for all $i$, where 
the slope $\mu$ of a torsion free sheaf is defined by $\mu(\bullet) = \text{deg}(\bullet)/\text{rank}(\bullet)$
with $\text{deg}(\bullet) = c_1(\bullet)L^{\dim Y-1}$.

By a theorem of Mehta-Ramanathan \cite{MR}, if we take a general curve section $C$ 
(a complete intersection of general hypersurfaces of sufficiently high degrees),
then the restriction $0 \subset N_{1,C} \subset \dots \subset N_{m,C} = V_C$ for 
$N_{i,C} = N_i \otimes \mathcal{O}_C$ is a Harder-Narasimhan filtration of $V_C = V \otimes \mathcal{O}_C$.
 
Let $U' = V/N_{m-1}$ be the last quotient.
Since $V$ is nef, we have $\text{deg}(U') \ge 0$.
If $\text{deg}(U') > 0$, then $V_C$ is an ample vector bundle on $C$ by Hartshorne \cite{Hartshorne}.
In this case, we have $V = W$.

Hence we may assume in  the following that $\text{deg}(U') = 0$.
In this case, we set $W : = N_{m-1}$.
Then $W_C = N_{m-1,C}$ is an ample vector bundle on $C$. 
We shall prove that  the quotient bundle $U'$ splits and there is a subbundle $U \cong U'$ 
which is a unitary flat bundle using the following general result
(a similar  argument is given on page 86 of \cite{cat-dett}).
Shigeharu Takayama communicated to the  authors that the following result was independently proved by 
Genki Hosono (cf. \cite{Hosono}):
 
\begin{Thm}\label{sum}
Let $X$ be a complex manifold which is not necessarily compact, and let 
\[
0 \to E \to F' \to G \to 0
\]
be an exact sequence of holomorphic vector bundles on $X$.
Assume that $F'$ has a smooth positive definite hermitian metric whose curvature form
$\Theta_{F'}$ is semi-positive, 
and that the curvature form $\Theta_G$ of the induced metric on $G$ vanishes.
Then the orthogonal complement $E^{\perp}$ of $E$ with respect to the metric on $F'$ 
is a holomorphic subbundle of $F'$ which is isomorphic to $G$.
In particular we have $F' \cong E \oplus^{\perp} G$ as holomorphic bundles with metric.
\end{Thm}

\begin{proof}
We shall prove the dual statement; namely let now
\[
\begin{CD}
0 @>>> S @>i>> F @>p>> Q @>>> 0
\end{CD}
\]
be an exact sequence of holomorphic vector bundles, where 
$F$ has a smooth positive definite hermitian metric whose curvature form
$\Theta_F$ is semi-negative, 
and that the curvature form $\Theta_S$ of the induced metric on the subbundle $S$ vanishes.
Then we shall prove that the orthogonal complement $S^{\perp}$ of $S$ in $F$ is a holomorphic subbundle 
which is isomorphic to the quotient bundle $Q$, hence $F \cong S \oplus^{\perp} Q$.

Let $D_F: A^0(F) \to A^1(F)$ be the connection of $F$ which is compatible with the complex structure 
and the metric, and
let $\s = pD_Fi: A^0(S) \to A^1(Q)$ be the second fundamental form.
Let $p_0: S^{\perp} \to Q$ be the restriction of $p$.
$p_0$ is an isomorphism of $C^{\infty}$ vector bundles with hermitian metrics.
Then $D_S = D_F - p_0^{-1}\s: A^0(S) \to A^1(S)$ and $D_Q = pD_Fp_0^{-1}: A^0(Q) \to A^1(Q)$ 
are the connections which are compatible with the complex structures and the induced metrics.  
We have the  formula 
\[
(\Theta_S u,v) = (\Theta_Fu,v) - (\s(u),\s(v))
\]
for $u,v \in A^0(S)$.
Since $\Theta_S = 0$ and $\Theta_F$ is semi-negative while the metric is positive definite, 
we deduce that $\s = 0$.
Thus $D_F(u) \in A^1(S)$ for $u \in A^0(S)$.
It follows that $D_F(u) \in A^1(S^{\perp})$ for $u \in A^0(S^{\perp})$.
If $u$ is a holomorphic section of $Q$, then $D_Q(u) \in A^{1,0}(Q)$.
It follows that $D_F(p_0^{-1}(u)) = p_0^{-1}D_Q(u) \in A^{1,0}(F)$.
Hence $p_0^{-1}(u)$ is a holomorphic section of $F$.
Therefore $S^{\perp}$ is a holomorphic subbundle.
\end{proof}

We can deduce from Theorem \ref{sum} the following well-known result:

\begin{Expl}
Let $F$ be a locally free sheaf of rank $2$ on an algebraic curve $C$ which is a non-trivial extension of
$\mathcal{O}_C$ by a nef invertible sheaf.
Then $F$ is nef but $F$ does not admit a smooth hermitian metric with semi-positive curvature.
\end{Expl}
 
We apply the above theorem to the restriction $V^o_C$, $(U')^o_C$ and $W^o_C$ 
of $V$, $U'$ and $W$, respectively, to the open part $C^o = C \cap Y^o$.
Since the numerical semi-positivity of $U'$ is a consequence of the semi-positivity of the curvature
and the null boundary contributions, the condition $\text{deg}(U')=0$ implies that 
the curvature of $(U')^o_C$ vanishes.
By Theorem~\ref{sum}, 
there is a holomorphic subbundle $U^o_C$ of $V^o_C$ such that 
$V^o_C = W^o_C \oplus^{\perp} U^o_C$, and such that
$U^o_C$ is flat with respect to the connection associated to the 
positive definite hermitian metric of $V^o_C$.
Thus $U^o_C$ is unitary flat with respect to the Gauss-Manin connection.

Since the local monodromies of the Gauss-Manin connection on $H^o_{\ZZ}$ are unipotent, 
the local monodromies of $U^o_C$ are trivial.
It follows that $U^o_C$ extends to a unitary flat vector bundle $U_C$ on $C$.
Since the curve sections $C$ are general, we can move $C$ around so that the unitary flat bundles 
$U_C$ can be extended to a unitary flat bundle $U^{oo}$ on a big open subset $Y^{oo}$ of $Y$, 
an open subset whose complement has codimension at least $2$ 
by virtue of the following argument.

The union of general curves which are complete intersections of hypersurface sections 
of high degree is a constructible set $Y^{*}$, 
which contains an open  subset of $Y$, else the curves are not general.
Moreover  the complement of $Y^{*}$  does not
contain an open  subset in any irreducible divisor $\Delta$ since  $C$ is general.
Thus  the complement of   $Y^*$ has codimension at least $2$. 
By the same token, at the general points of  $Y^*$ 
the tangents of the curves $C$ fill out  a big open subset.
Hence we get a unitary flat bundle $U^{oo}$ on a big open set $Y^{oo}$. 
Since $\pi_1(Y^{oo}) \cong \pi_1(Y)$, $U^{oo}$ extends to a unitary flat bundle on the whole $Y$.

On the other hand, the canonical extension $H$ of $H^o$ is characterized as a holomorphic
vector bundle by the condition that the norms of its holomorphic sections have at most logarithmic growth.
Therefore $U$ is still a holomorphic subbundle of $V$, and we have a direct sum 
decomposition $V = W \oplus U$.
\qed

\begin{Rem}\label{reflexive}
Keeping the same assumptions of  Theorem \ref{main}, 
further use of Theorem \ref{sum} shows:
if there is a generically surjective homomorphism $V \ra Q$  to a reflexive sheaf
with  $\text{deg}(Q)=0$, then  there is a direct sum decomposition 
$V = W' \oplus Q'$ with $Q' \cong Q$ being a unitary flat bundle with respect to the natural hermitian metric
(see also the argument in the next section).
\end{Rem}

\section{ Proof of Theorem~\ref{full}} 

In this section we fully extend Fujita's decomposition theorem to the case 
of a higher dimensional base variety $Y$, proving Theorem~\ref{full}.

In the proof we shall use a similar notation to the one of Theorem~\ref{main}.
We denote by $D$ the degeneracy  locus of $f$, 
the locus over which $f$ is not smooth (also called  the set of critical values of $f$), 
so that $f$ is smooth over $Y^o = Y \setminus D$.

We take a log resolution $u: Y' \to Y$ for $D \subset Y$; $Y'$ is smooth and projective,
$u$ is birational, and $D' = u^{-1}(D)$ is a simple normal crossing divisor.
Since the direct image sheaf $V$ does not depend on the choice of a birational model of $X$,
we may assume that there is a morphism $f: X \to Y'$ such that $f = u \circ f'$.

Let $V' = f'_*\mathcal{O}_X(K_X - (f')^*K_{Y'})$.
It is a locally free sheaf by theorem 4.1 of \cite{viehweg}.
Moreover, the weak-positivity of $V'$ and $V$ was proven by Viehweg (theorem III of \cite{viehweg}).

Since $K_{Y'} \ge u^*K_Y$, we have $u_*V' \subset V$.
More precisely,  if $R$ is the ramification divisor of $u$, $K_{Y'} = u^* (K_Y) + R$, we have: 
\[
V =  f_*(\mathcal{O}(K_X))(-K_Y) =u_*(f'_*(\mathcal{O}(K_X))(-K_{Y'})+R) \supset u_*(V').
\]

On the other hand, since these sheaves are equal  outside of codimension  $2$
and $V$  is torsion free, $V$ embeds into 
the reflexive hull of $u_*V'$: 
$V \subset (u_*V')^{\vee \vee}$.

Our proof of (2) of Theorem~\ref{full} is similar to that of Theorem~\ref{main} in the following way.
We take a Harder Narasimhan filtration of $V$, let $U'$ be the last quotient, 
and assume that $\text{deg}(U') = 0$.
By Mehta-Ramanathan, we take a generic curve $C$ and consider the restriction $U'_C$.
We recall that $\text{deg}(U'_C)$ is the sum of the integration of the curvature on the open part 
$C^o = C \cap Y^o$ and the sum of the boundary contributions, 
as shown in lemma 21, page 268  of \cite{K1}. 
Indeed the line bundle $L : = \Lambda^{top} ( U'|_C)$ on $C$ has a smooth metric outside a finite set of points $P$,  
where the growth of  the norm of a local generating section $s_P$
is asymptotic to $|s_P| \sim |t_P|^{-\al_P} (\log |t_p|)^{\beta_P}$, where the $\al_P$ are 
non-negative rational numbers,
in view of the choice of the branch for the canonical extension. 
We observe that each term is non-negative, 
due to the semipositivity of the curvature.
Since $\text{deg}(U'_C) = 0$, both contributions vanish, i.e., 
all the numbers $\al_P =0 $, and the curvature vanishes.
  
Since $\al_P =0 $, the  local monodromies are unipotent whence trivial (unipotent and  unitary).
Indeed, if the local monodromy  is not unipotent, then the growth of the  norm of the generating holomorphic sections, defined by using the multivalued flat sections, 
is not logarithmic, so for some of them it is greater than a constant times 
$|t_P|^{-a_P} (\log |t_p|)^{b_P}$ with $a_P > 0$, contradicting that the sum of the $a_P$'s equals $\al_P =0 $.
 
By Theorem~\ref{sum}, there is a holomorphic subbundle $U^o_C \cong (U'_C)_{C^o}$ such that 
$V^o_C = W^o_C \oplus^{\perp} U^o_C$ on $C^o$.
Since $U^o_C$ is unitary flat with respect to the given hermitian metric, and with trivial local monodromies
there is a locally free sheaf with a unitary flat connection $U_C$ on $C$ which 
is an extension of $U^o_C$.    
By the same argument as in the proof of Theorem~\ref{main}, we obtain a locally free sheaf $U$ on $Y$
which is unitary flat.

We have to prove that $U$ is a direct summand of $V$.
By the same argument applied to $f': X \to Y'$, we obtain a locally free sheaf $U_{Y'}$ on $Y'$
which is unitary flat with respect to the given metric.
Since $u^*U$ and $U_{Y'}$ are defined as extensions of the same unitary flat bundles on $Y^o$, 
we have $u^*U = U_{Y'}$.
This time, since $V'$ is a subbundle of the canonical extension defined by using the local monodromies, 
$U_{Y'}$ is a direct summand of $V'$; $V' = U_{Y'} \oplus^{\perp} W_{Y'}$.
By taking the direct image by $u$, we have 
\[
U \oplus^{\perp} u_*W_{Y'} \subset V \subset U \oplus^{\perp} (u_*W_{Y'})^{\vee \vee}.
\]
Therefore we have $V = U \oplus^{\perp} W$.

\qed 
  
\begin{proof}[Proof of Proposition~\ref{orb}]
In the general case, we  construct the following  bimeromorphically Cartesian diagram
using \cite{K1}~Theorem~17:

\begin{equation*}
\xymatrix{
X'' \ar[d]^{f''}\ar[r]^{v'} & X \ar[d]^{f'}\ar\ar[r]^{Id}  & X\ar[d]^f\\
Y'' \ar[r]^{u'} &Y'  \ar[r]^u & Y ,}
\end{equation*}

\noindent
where $ X'', Y', Y''$ are all smooth,  such that 

(i)  $u$ is birational, and the degeneracy locus of $f'$ is contained in a  simple normal crossing divisor $D'$ on $Y'$.

(ii) $u'$ is a finite surjective morphism with Abelian Galois group $G$,  
and $f'' : X'' \ra Y''$ satisfies the hypotheses of Theorem \ref{main}.

Denote by $V'',  V', V$ the respective images of the relative canonical sheaves.
By \cite{viehweg} or by the definition of the canonical extension, 
we have a generically invertible homomorphism 
between the  two vector bundles of the same rank
$V'' \ra u' {^*} (V')$.
Thus we have an injective homomorphism $(u'_*V'')^G \to V'$ between locally free sheaves 
which is generically surjective.

By Theorem~\ref{main}, 
there is an orthogonal splitting $V'' = U'' \oplus W''$, where $U''$ is unitary flat and $W''$ 
is generically ample.
Let $U'{^{\text{orb}}}$ be smallest subsheaf of $V'$ containing $(u'_*U'')^G$ and such that 
$V'/U'{^{\text{orb}}}$ is torsion free, i.e. the hull of $(u'_*U'')^G$ in $V'$.
By the definition of the canonical extension, $U'{^{\text{orb}}}$ coincides with the canonical extension of 
$(u'_*U'')^G \vert_{Y' \setminus D'}$, hence a locally free sheaf.
We finish the  proof  by taking $U^{\text{orb}}$ to be the hull of $u_*(u'_*U'')^G$ in $V$.
\end{proof}


\end{document}